\documentclass[12pt]{amsart}

\usepackage{amsmath}
\usepackage{amsmath,amsfonts,amssymb,amsthm}
\usepackage{graphicx,color}
\DeclareMathRadical{\sqrtsign}{symbols}{"70}{largesymbols}{"70}
\providecommand{\abs}[1]{\lvert#1\rvert}

\newlength{\figboxwidth}             
\newcommand{\infinity}{\infty}

\def\@ifundefined#1#2#3%
  {\expandafter\ifx\csname#1\endcsname\relax#2\else#3\fi}

\@ifundefined{theoremstyle}{
}{
\theoremstyle{plain} 
}
\newtheorem{theorem}{Theorem}[section]

\newtheorem{proposition}[theorem]{Proposition}
\newtheorem{lemma}[theorem]{Lemma}

\newtheorem{corollary}[theorem]{Corollary}

\@ifundefined{theoremstyle}{
}{
\theoremstyle{definition} 
}

\newtheorem{remark}[theorem]{Remark}

\mathchardef\GG="321D
\newcommand{\mcc}[1]{{}}

\numberwithin{equation}{section}

\title[Dynamical spectra for skew products with unit circle fiber]
{Dynamical spectra for skew products with unit circle fiber}

\author{Carlos Gustavo Moreira}   

\address{Carlos G. Moreira: 
IMPA, Instituto de Matem\'atica Pura e Aplicada, Rio de Janeiro, 22460320, Brazil.
}
\email{gugu@impa.br}

\author{Christian Camilo Silva Villamil}  

\address{Christian Camilo Silva Villamil: 
 Shenzhen International Center for Mathematics, Southern University of Science and Technology, Shenzhen, 518000, China.}
\email{ccsilvav@sustech.edu.cn}

\author{Raúl Ures}
\address{Raúl Ures: Department of Mathematics and Shenzhen International Center for Mathematics, Southern University of Science and Technology, Shenzhen, 518000, China.}
\email{ures@sustech.edu.cn}


\begin{document}

\begin{abstract}
The dynamical Markov and Lagrange spectra are subsets of the real line widely studied and that share some similarities with the classical spectra, e.g. typical dynamical spectra, associated to horseshoes with Hausdorff dimension greater than one, have nonempty interior and, in the conservative setting, the intersections of the spectra with half-lines determine a continuous function. Here we study some similar questions in the context of skew products of surface diffeomorphisms with unit circle fiber. 
\end{abstract}

\maketitle

\tableofcontents

\section{Introduction}\label{s.introduction}

The classical Lagrange and Markov spectra are closed subsets of the real line related to Diophantine approximations. They arise naturally in the study of rational approximations of irrational numbers and of indefinite binary quadratic forms, respectively. 

It is worth to point out here that the Lagrange and Markov spectra can be defined in the following \emph{dynamical} way in terms of the continued fraction algorithm: denote by $[a_0,a_1,\dots]$ the continued fraction $a_0+\frac{1}{a_1+\frac{1}{\ddots}}$. Let $\Sigma=\mathbb{N}^{\mathbb{Z}}$ be the space of bi-infinite sequences of positive integers, $\sigma:\Sigma\to\Sigma$ be the left-shift map given by  $\sigma((a_n)_{n\in\mathbb{Z}}) = (a_{n+1})_{n\in\mathbb{Z}}$, and let $f:\Sigma\to\mathbb{R}$ be the function
$$f((a_n)_{n\in\mathbb{Z}}) = [a_0, a_1,\dots] + [0, a_{-1}, a_{-2},\dots].$$
Then, 
$$\mathcal{L}=\left\{\limsup_{n\to\infty}f(\sigma^n(\theta))<\infty:\theta\in\Sigma\right\} \quad \textrm{and} \quad \mathcal{M}= \left\{\sup_{n\to\infty}f(\sigma^n(\theta))<\infty:\theta\in\Sigma\right\}.$$

The reader can find more information about the structure of these sets in the classical book \cite{CF} of Cusick and Flahive, but let us mention here that:
\begin{itemize}
\item Markov showed that $\mathcal{L}\cap(-\infty, 3)=\mathcal{M}\cap(-\infty, 3)=\{\sqrt{9-4/z_n^2}:n\in \mathbb{N}\}$ where $z_n$ are the \emph{Markov numbers}, that is, the largest coordinate of a triple $(x_n,y_n,z_n)\in \mathbb{N}^3$ verifying the Markov equation 
$$x_n^2+y_n^2+z_n^2=3x_ny_nz_n.$$
\item Hall showed that $\mathcal{L}$ ($\subset\mathcal{M}$) contains a half-line and Freiman determined the biggest half-line contained in the spectra, namely $[c,+\infty)$ where 
$$c=\frac{2221564096+283748\sqrt{462}}{491993569}\simeq 4.52782956\dots$$ 
\item Moreira proved in \cite{M3} several results on the geometry of the Markov and Lagrange spectra, for example, that the map $d:\mathbb{R} \rightarrow [0,1]$, given by
$$d(t)=HD(\mathcal{L}\cap(-\infty,t))= HD(\mathcal{M}\cap(-\infty,t)),$$
(where $HD(X)$ denotes the Hausdorff dimension of the set $X$) is continuous, surjective and such that $\max\{t\in\mathbb{R}:d(t)=0\}=3.$
\end{itemize}

In the sequel, we will consider the natural generalization of the classical Lagrange and Markov spectra in the context of smooth diffeomorphisms of compact riemannian manifolds: let $\Phi:M\rightarrow M$ be a diffeomorphism of a $C^{\infty}$ compact riemannian manifold $M$ and let $F:M\rightarrow \mathbb{R}$ be a continuous function. Following the above characterization of the classical spectra, we define the maps $\ell_{\Phi,F}: M \rightarrow \mathbb{R}$ and $m_{\Phi,F}: M \rightarrow \mathbb{R}$ given by $\ell_{\Phi,F}(x)=\limsup_{n\to \infty}F(\Phi^n(x))$ and $m_{\Phi,F}(x)=\sup_{n\in\mathbb{Z}}F(\Phi^n(x))$ for $x\in M$. Call $\ell_{\Phi,F}(x)$ the \textit{Lagrange value} of $x$ associated to $F$ and $\Phi$ and $m_{\Phi,F}(x)$ the \textit{Markov value} of $x$ associated to $F$ and $\Phi$. Given $X\subset M$, closed and $\Phi$-invariant, the sets 
$$\mathcal{L}_{X,\Phi,F}=\ell_{\Phi,F}(X)=\{\ell_{\Phi,F}(x):x\in X\}$$
and
$$\mathcal{M}_{X,\Phi,F}=m_{\Phi,F}(X)=\{m_{\Phi,F}(x):x\in X\}$$
are respectively called \textit{Lagrange Spectrum} of $(X,\Phi,F)$ and \textit{Markov Spectrum} of $(X,\Phi,F)$ . As in the classical spectra, it can be easily verified that $\mathcal{L}_{X,\Phi,F}\subset \mathcal{M}_{X,\Phi,F}$.

In this article, we will only consider the cases when $M=S$ is a compact surface, $\Phi=\varphi$ is a diffeomorphism of $S$ and $X=\Lambda$ is a horseshoe\footnote{i.e., a non-empty compact invariant hyperbolic set of saddle type which is transitive, locally maximal, and not reduced to a periodic orbit (cf. \cite{PT93} for more details).} for $\varphi$; or the case when $M=S\times \mathbb{S}^1$, the diffeomorphism $\Phi:S\times \mathbb{S}^1\rightarrow S\times \mathbb{S}^1$ is given by $\Phi(x,t)=(\varphi(x),R_x(t))$, where for $x\in M$, $R_x:\mathbb{S}^1\rightarrow \mathbb{S}^1$ is a diffeomorphism and $X=\Lambda\times \mathbb{S}^1$ (here $S$, $\varphi$ and $\Lambda$ are as before).

Given a diffeomorphism $\varphi:S\rightarrow S$ of a compact surface $S$ with a horseshoe $\Lambda$, let us remember the combinatorial nature of $\varphi|_{\Lambda}$. Fix a Markov partition $\{R_a\}_{a\in \mathcal{A}}$ of $\Lambda$ with sufficiently small diameter consisting of rectangles $R_a \sim I_a^s \times I_a^u$ delimited by compact pieces $I_a^s$, $I_a^u$, of stable and unstable manifolds of certain points of $\Lambda$, see \cite{PT93} Theorem 2, page 172. The set $\mathcal{B}\subset \mathcal{A}^{2}$ of admissible transitions consist of pairs $(a,b)$ such that $\varphi(R_a)\cap R_{b}\neq \emptyset$; so, we can define the transition matrix $B$ by
$$b_{ab}=1 \ \ \text{if} \ \  \varphi(R_a)\cap R_b\neq \emptyset \ \ \text{and}  \ b_{ab}=0  \ \text{otherwise, for $(a,b)\in \mathcal{A}^{2}$.}$$

Consider the space $\Sigma_{\mathcal{A}}=\left\{\underline{a}=(a_{n})_{n\in \mathbb{Z}}:a_{n}\in \mathcal{A} \ \text{for all} \ n\in \mathbb{Z}\right\}$ with the metric $d(\underline{a},\underline{b})=\max\{2^{-n^+_{\underline{a},\underline{b}}},2^{-n^-_{\underline{a},\underline{b}}} \}$ where $n^-_{\underline{a},\underline{b}}=\max\{n\in \mathbb{N}:a_{-n}=b_{-n} \}$ and $n^+_{\underline{a},\underline{b}}=\max\{n\in \mathbb{N}:a_{n}=b_{n} \}$ and consider the homeomorphism of $\Sigma_{\mathcal{A}}$, the shift, $\sigma:\Sigma_{\mathcal{A}}\to\Sigma_{\mathcal{A}}$ defined by $\sigma(\underline{a})_{n}=a_{n+1}$. Let $\Sigma_{\mathcal{B}}=\left\{\underline{a}\in \Sigma_{\mathcal{A}}:b_{a_{n}a_{n+1}}=1\right\}$; this set is a closed and $\sigma$-invariant subspace of $\Sigma_{\mathcal{A}}$. Still denote by $\sigma$ the restriction of $\sigma$ to $\Sigma_{\mathcal{B}}$. The pair $(\Sigma_{\mathcal{B}},\sigma)$ is a subshift of finite type, see \cite{{Shub}} chapter 10. The dynamics of $\varphi$ on $\Lambda$ is topologically conjugate to the sub-shift $\Sigma_{\mathcal{B}}$, namely, there is a homeomorphism $\Pi: \Lambda \to \Sigma_{\mathcal{B}}$, that is bi-Hölder, such that $\varphi\circ \Pi=\Pi\circ \sigma$.

It turns out that dynamical Markov and Lagrange spectra associated to hyperbolic dynamics are closely related to the classical Markov and Lagrange spectra. Several results on the Markov and Lagrange dynamical spectra associated to horseshoes in dimension 2 which are analogous to previously known results on the classical spectra were obtained recently: In \cite{MR2} it is shown that typical dynamical spectra associated to horseshoes with Hausdorff dimensions larger than one have non-empty interior (as the classical ones). In \cite{M4} it is shown that typical Markov and Lagrange dynamical spectra associated to horseshoes have the same minimum, which is an isolated point in both spectra and is the image by the function of a periodic point of the horseshoe. 

Moreira's theorem of \cite{M3} was generalized first in \cite{CMM16} in the context of {\it conservative} diffeomorphism with some horseshoe with Hausdorff dimension smaller than $1$ and later, the condition on the dimension of the horseshoe was removed in \cite{GCD}. More specifically, the authors proved that for typical choices of the dynamics and of the real function, the intersections of the corresponding dynamical Markov and Lagrange spectra with half-lines $(-\infty,t)$ have the same Hausdorff dimension, and this defines a continuous function of $t$.

Here we study some similar questions in the context of skew products of surface diffeomorphisms with unit circle fiber. However, in our main theorem we suppose that the circle diffeomorphism is an irrational rotation with angle that does not depend on the point in the surface. This theorem is quite related to the result of the previous paragraph on the interior of the spectra:

\begin{theorem}\label{principal1}
Let $\Lambda$ be a horseshoe of $\varphi:S\rightarrow S$ and $r\geq 2$, $\alpha\in \mathbb{R}\setminus \mathbb{Q}$ fixed. If $\varphi\times {\text{rot}}_{\alpha}:S\times \mathbb{S}^1\rightarrow S\times \mathbb{S}^1$ is given by $(\varphi\times {\text{rot}}_{\alpha})(x,t)=(\varphi(x),{\text{rot}}_{\alpha}(t))$, where ${\text{rot}}_{\alpha}$ is the rotation with angle $\alpha$, then there exists an open and dense set $\mathcal{R}_{\varphi,\Lambda}\subset C^r(S\times \mathbb{S}^1,\mathbb{R})$ such that for each $F\in \mathcal{R}_{\varphi,\Lambda}$:
 $$int( \mathcal{L}_{\Lambda\times \mathbb{S}^1,\varphi\times {\text{rot}}_{\alpha},F})\neq \emptyset.$$ 

Moreover, for $F\in\mathcal{R}_{\varphi,\Lambda}$ there is only one point of maximum $(\tilde{x},\tilde{t})$ and we can find a sequence of non-trivial intervals contained in $\mathcal{L}_{\Lambda\times \mathbb{S}^1,\varphi\times R_{\alpha},F}$ converging to the point of maximum $F(\tilde{x},\tilde{t})$. If $\tilde{x}$ is not periodic there is a non-trivial interval with right endpoint  $F(\tilde{x},\tilde{t})$ that is contained in $\mathcal{L}_{\Lambda\times \mathbb{S}^1,\varphi\times R_{\alpha},F}$.
    
\end{theorem}

\begin{remark}
Theorem \ref{principal1} will be a consequence of propositions \ref{rphi1} and \ref{rphi2} that apply to more general skew products. 
\end{remark}

\section{Proof of the theorems}
Before proceeding with the proof of the theorems, let us see first if it is possible a reduction of the spectra when $M=S\times \mathbb{S}^1$ to the case when $M=S$. It will also shows some of the techniques that will be used in the proof of Theorem \ref{principal1}.

Given $r\geq 2$ and $F\in C^r(S\times \mathbb{S}^1,\mathbb{R})$ we can consider the function $f_F:S\rightarrow \mathbb{R}$ given by $f_F(x)=\max_{t\in \mathbb{S}^1}F(x,t)$. Note that $f_F$ is Lipschitz continuous: since $F$ is differentiable, there exists some constant $c>0$ such that given $x_1,x_2\in S$ and $t\in \mathbb{S}^1$ one has $\abs{F(x_1,t)-F(x_2,t)}\leq c\cdot d(x_1,x_2)$, thus $F(x_1,t)-F(x_2,t)\leq c\cdot d(x_1,x_2)$. From this we get $F(x_1,t)\leq c\cdot d(x_1,x_2)+ f_F(x_2)$ and then $f_F(x_1)\leq c\cdot d(x_1,x_2)+ f_F(x_2)$. By symmetry, the other inequality also holds and then $\abs{f_F(x_1)-f_F(x_2)}\leq c\cdot d(x_1,x_2)$.

With some additional conditions one can say a little bit more:

\begin{proposition}\label{differentiable}
Let $F\in C^r(S\times \mathbb{S}^1,\mathbb{R})$ and suppose that given $x\in S$ there exists a unique $t_x\in \mathbb{S}^1$ such that $f_F(x)=F(x,t_x)$ and $\frac{\partial^2F}{\partial t^2}(x,t_x)<0$. Then, $f_F\in C^{r-1}(S,\mathbb{R})$.
\end{proposition}

\begin{proof}
Given $x\in S$ consider $\epsilon_0>0$ such that for $(y,s)\in B(x,\epsilon_0)\times (t_x-\epsilon_0,t_x+\epsilon_0)$ one has $\frac{\partial^2F}{\partial t^2}(y,s)<0$. Let $\epsilon_1=F(x,t_x)-\max_{t\in \mathbb{S}^1\setminus (t_x-\epsilon_0,t_x+\epsilon_0)}F(x,t)>0$ and consider $0<\epsilon_2<\epsilon_0$ such that for $y\in B(x,\epsilon_2)$ and $s\in \mathbb{S}^1$ one has $\abs{F(y,s)-F(x,s)}<\epsilon_1/2$. We then conclude for $y\in B(x,\epsilon_2)$, $t_y\in (t_x-\epsilon_0,t_x+\epsilon_0)$. Indeed, for $s\in \mathbb{S}^1\setminus (t_x-\epsilon_0,t_x+\epsilon_0)$
$$F(y,s)<F(x,s)+\epsilon_1/2\leq F(x,t_x)-\epsilon_1+\epsilon_1/2=F(x,t_x)-\epsilon_1/2<F(y,t_x).$$ 
In addition, for $y\in B(x,\epsilon_2)$ there is only one $s\in (t_x-\epsilon_0,t_x+\epsilon_0)$ ($s=t_y$) such that $\frac{\partial F}{\partial t}(y,t)=0$, since for $(y,s)\in B(x,\epsilon_2)\times (t_x-\epsilon_0,t_x+\epsilon_0)$, $\frac{\partial^2F}{\partial t^2}(y,s)<0$ (here we used Rolle's theorem). The implicit function theorem applied to $G(x,t)=\frac{\partial F}{\partial t}(x,t)$ let us conclude in some neighborhood of $x$ that the function $y\rightarrow t_y$ is of class $C^{r-1}$ and then $f_F(\cdot)=F(\cdot,t_{\cdot})$ is also of class $C^{r-1}$.    
\end{proof}

Given a horseshoe $\Lambda\subset S$ of a diffeomorphism $\varphi:S\rightarrow S$ and an irrational number $\alpha$ we may ask if some of the following equalities  $\mathcal{L}_{\Lambda\times \mathbb{S}^1,\varphi\times {\text{rot}}_{\alpha},F}=\mathcal{L}_{\Lambda,\varphi,f_F}$ and $   \mathcal{M}_{\Lambda\times \mathbb{S}^1,\varphi\times {\text{rot}}_{\alpha},F}=\mathcal{M}_{\Lambda,\varphi,f_F}$ are true. To see that in general this is not the case, let us introduce some sets. First, denote by $\mathcal{R}$ to the set of non-constant functions of class $C^r$ defined in $\mathbb{S}^1$, which is clearly open and dense.

Given $x\in \Lambda$ we will write the symbolic representation of $x$ in the way $\Pi(x)=(\dots,a_{-2},a_{-1};a_0,a_1, a_2\dots)$, where the letter $a_0$ following to $;$ is in the $0$ position of the sequence. Additionally, given an admissible finite sequence $A=(a_{-s},...,a_s)\in \mathcal{A}^{2s+1}$ (i.e., $(a_i,a_{i+1})\in \mathcal{B}$ for all $-s\le i<s$), define $R_A=\Pi^{-1}\{(x_{n})\in \Sigma_{\mathcal{B}}:(x_{-s},\dots,x_0,\dots,x_s)=A\}$.

\begin{lemma}
Given a horseshoe $\Lambda\subset S$ with $HD(\Lambda)<1/2$, there exists a residual subset $\mathcal{R}_{\Lambda}\subset C^r(S,\mathbb{R})$ such that for any $f\in \mathcal{R}_{\Lambda}$, $f|_{\Lambda}$ is injective. 
\end{lemma}
\begin{proof}
Given $n\in \mathbb{N}$, let $\mathcal{R}^n_{\Lambda}\subset C^r(S,\mathbb{R})$ be the set of functions $f$ such that if $X_1,\dots, X_s$ are all the admissible words of size $2n+1$, then one has $f(R_{X_i})\cap f(R_{X_j})= \emptyset$, for $i\neq j$. This set is clearly open because the rectangles $\{R_{X_i}\}_{i=1}^s$ are compact and disjoint. It is also dense because given a function $f\in C^r(S,\mathbb{R})$ if $h:\mathbb{R}^2\rightarrow \mathbb{R}$ is given by $h(x,y)=x-y$, $B$ denotes the usual \textit{box counting dimension} and $c\in\mathbb{R}$, then for the set $R_{X_i,X_j,c}=h((f(R_{X_i})+c)\times f(R_{X_j }))$ one has
\begin{eqnarray*}
HD(R_{X_i,X_j,c})&\leq& B(R_{X_i,X_j,c})\leq B((f(R_{X_i})+c)\times f(R_{X_j}))\\ &\leq& B(f(R_{X_i})+c)+B(f(R_{X_j}))\leq 2B(f(\Lambda))\\ &\leq &2B(\Lambda)=2HD(\Lambda)<1  \end{eqnarray*}
and then 
$$\{ t\in \mathbb{R}:(f(R_{X_i})+c)\cap(f(R_{X_j})+t)= \emptyset \}=\mathbb{R}\setminus R_{X_i,X_j,c}$$
has Lebesgue measure $1$. Thus, we can inductively add to $f$ a small constant in the neighborhood of each of the sets $R_{X_i}$: we start with $t_1=0$ and once we have chosen $t_1,\dots,t_k$ with $k<s$, we take a small  $t_{k+1}\in \bigcap_{i=1}^k (\mathbb{R}\setminus R_{X_i,X_{k+1},t_i})$. Finally, define the set $\mathcal{R}_{\Lambda}=\bigcap_{n\in\mathbb{N}}\mathcal{R}^n_{\Lambda}$, this set works because given two different points $x_1,x_2\in \Lambda$, there exists some $n_0\in \mathbb{N}$ such that the words of size $2n_0+1$ in the center of their kneading sequences are different. This finishes the proof of the lemma.

\end{proof}

Given $f\in \mathcal{R}_{\Lambda}$ denote by $x_f$ to the only one element in $\Lambda$ such that $f(x_f)=\max f|_{\Lambda}$.

\begin{theorem}
Fix a horseshoe $\Lambda\subset S$ with $HD(\Lambda)<1/2$. Given $(f,g)\in \mathcal{R}_{\Lambda}\times \mathcal{R}$ there exists a constant $\alpha_{f,g}>0$ such that the function $F_{f+g}:S\times \mathbb{S}^1\rightarrow \mathbb{R}$ given by $F_{f+g}(x,t)=f(x)+\alpha_{f,g}\cdot (g(t)-\min g)$ satisfies for $\alpha\in \mathbb{R}\setminus \mathbb{Q}$ that 
$$\mathcal{L}_{\Lambda\times \mathbb{S}^1,\varphi\times {\text{rot}}_{\alpha},F_{f+g}}\neq\mathcal{L}_{\Lambda,\varphi,f_{F_{f+g}}} \ \ \text{and} \ \ \mathcal{M}_{\Lambda\times \mathbb{S}^1,\varphi\times {\text{rot}}_{\alpha},F_{f+g}}\neq\mathcal{M}_{\Lambda,\varphi,f_{F_{f+g}}}.$$
\end{theorem}

\begin{proof}
As in the proof of the previous lemma
\begin{eqnarray*}
HD(h(\mathcal{M}_{\varphi,f}\times\mathcal{M}_{\varphi,f}))&\leq& B(h(\mathcal{M}_{\varphi,f}\times\mathcal{M}_{\varphi,f}))\leq B(\mathcal{M}_{\varphi,f}\times\mathcal{M}_{\varphi,f})\leq 2B(\mathcal{M}_{\varphi,f})\\ &\leq& 2B(f(\Lambda))\leq 2B(\Lambda)=2HD(\Lambda)<1.    
\end{eqnarray*}
Then Lebesgue almost every point $\beta\in \mathbb{R}$ satisfies $(\mathcal{M}_{\varphi,f}+\beta)\cap \mathcal{M}_{\varphi,f}=\emptyset$.

Consider a periodic point $\Pi^{-1}(\bar{R})$ where $R$ is admissible, $\bar{R}=\dots R;R\dots$ is the infinite periodic sequence with period $R$, and $x_f\notin \mathcal{O}(\Pi^{-1}(\bar{R}))$. Additionally, take $z\in W^s(\Pi^{-1}(\bar{R}))\cap W^u(\Pi^{-1}(\bar{R}))$ close enough to $x_f$ ($f(z)>\max_{n\in \mathbb{Z}}f(\varphi^n(\Pi^{-1}(\bar{R})))$). Note that $z$ is not periodic because, in other case, the value $f(z)$ would be obtained close to the orbit of $\Pi^{-1}(\bar{R})$. As $f|_{\Lambda}$ is injective, we can find $i_0\in \mathbb{Z}$ such that $f(\varphi^{i_0}(z))>\sup_{i\neq i_0}f(\varphi^i(z))$. Then if we choose $\beta\in (0,f(\varphi^{i_0}(z))-\sup_{i\neq i_0}f(\varphi^i(z)))$, we can set $\alpha_{f,g}=\frac{\beta}{\max g-\min g}$. Thus, if $\bar{t}\in \mathbb{S}^1$ is such that $g(\bar{t})=\min g$, we have:
$$F_{f+g}(\varphi^{i_0}(z),\bar{t})=f(\varphi^{i_0}(z))+\alpha_{f,g}\cdot (g(\bar{t})-\min g)=f(\varphi^{i_0}(z))$$
and for $n\in \mathbb{Z}\setminus\{ 0\}$
\begin{eqnarray*}
F_{f+g}((\varphi\times {\text{rot}}_{\alpha})^n(\varphi^{i_0}(z),\bar{t}))&=&F_{f+g}(\varphi^{n+i_0}(z),{\text{rot}}_{\alpha}^n(\bar{t}))\\&=&f(\varphi^{n+i_0}(z))+\alpha_{f,g}\cdot (g({\text{rot}}_{\alpha}^n(\bar{t}))-\min g)\\ &\leq&  f(\varphi^{n+i_0}(z))+\alpha_{f,g}\cdot (\max g-\min g)=f(\varphi^{n+i_0}(z))+\beta\\ &<& f(\varphi^{n+i_0}(z))+f(\varphi^{i_0}(z))-\sup_{i\neq i_0}f(\varphi^i(z)) \leq f(\varphi^{i_0}(z)).
\end{eqnarray*}
From this we conclude that
$$m_{\varphi\times {\text{rot}}_{\alpha},F_{f+g}}(\varphi^{i_0}(z),\bar{t})=f(\varphi^{i_0}(z))=m_{\varphi,f}(z).$$
Additionally, given $x\in S$ one has
$$f_{F_{f+g}}(x)=\max _{t\in \mathbb{S}^1}F_{f+g}(x,t)=f(x)+\alpha_{f,g}\cdot (\max g-\min g)=f(x)+\beta,$$
which let us conclude that $\mathcal{M}_{\Lambda,\varphi,f_{F_{f+g}}}=\mathcal{M}_{\Lambda,\varphi,f}+\beta$. Therefore 
$$m_{\varphi\times {\text{rot}}_{\alpha},F_{f+g}}(\varphi^{i_0}(z),\bar{t})=m_{\varphi,f}(z)\notin \mathcal{M}_{\Lambda,\varphi,f}+\beta=\mathcal{M}_{\Lambda,\varphi,f_{F_{f+g}}},$$
that prove $\mathcal{M}_{\Lambda\times \mathbb{S}^1,\varphi\times {\text{rot}}_{\alpha},F_{f+g}}\neq\mathcal{M}_{\Lambda,\varphi,f_{F_{f+g}}}$.

Moreover, we have $m_{\varphi,f}(z)  \in\mathcal{L}_{\Lambda\times \mathbb{S}^1,\varphi\times {\text{rot}}_{\alpha},F_{f+g}}$. First observe that given $n,r\in \mathbb{N}$ and $s\in \mathbb{S}^1$, since $\alpha\in \mathbb{R}\setminus \mathbb{Q}$, the set $\{{\text{rot}}_{\alpha}^{k\cdot n+r}(s) \}_{k\in\mathbb{N}}=\{{\text{rot}}_{n\cdot \alpha}^k({\text{rot}}_{\alpha}^r(s) ) \}_{k\in\mathbb{N}}$ is dense in $\mathbb{S}^1$. Thus, if $j> 0$ is such that $\Pi(\varphi^{-j}(z))=R^{\infinity};HR^{\infinity}$, where $RHR$ is admissible, we can define the point 
$$\tilde{z}=\Pi^{-1}(R^{\infinity};HR^{n_1}HR^{n_2}H\dots R^{n_{k-1}}HR^{n_k}\dots ),$$ 
where the sequence $\{n_s\}_{s\in\mathbb{N}}$ is taken in such way that:
\begin{itemize}
    \item $n_s\rightarrow\infinity$
    \item ${\text{rot}}_{\alpha}^{(\Sigma_{i=1}^{s}n_i)\cdot \abs{R}+s\cdot \abs{H}}(\bar{t})\rightarrow {\text{rot}}_{\alpha}^{-j-i_0}(\bar{t}).$
\end{itemize}    
From this it follows that
$$\limsup_{n\rightarrow\infinity}F_{f+g}((\varphi\times {\text{rot}}_{\alpha}))^n(\tilde{z},\bar{t}))= \limsup_{n\rightarrow\infinity} F_{f+g}(\varphi^n(\tilde{z}), {\text{rot}}^n_{\alpha}(\bar{t}))=m_{\varphi,f}(z)$$
since for $n_s$ big, in the positions different than the position $i_0+j$ of $H$ (corresponding to $\varphi^{i_0}(z)$), the value of $F_{f+g}$ is smaller than $f(\varphi^{i_0}(z))$ because when we estimated for $n\in \mathbb{Z}\setminus\{ 0\}$, $F_{f+g}((\varphi\times {\text{rot}}_{\alpha})^n(\varphi^{i_0}(z),\bar{t}))< f(\varphi^{i_0}(z))$ we did not use the time $t=\bar{t}$. And we also have that 
\begin{eqnarray*}
&&\lim _{s\rightarrow\infinity}F_{f+g}((\varphi\times {\text{rot}}_{\alpha})^{(\Sigma_{i=1}^{s}n_i)\cdot \abs{R}+s\cdot \abs{H}+j+i_0}(\tilde{z},\bar{t}))\\&=&\lim _{s\rightarrow\infinity}F_{f+g}(\varphi^{(\Sigma_{i=1}^{s}n_i)\cdot \abs{R}+s\cdot \abs{H}+j+i_0}(\tilde{z}),{\text{rot}}_{\alpha}^{(\Sigma_{i=1}^{s}n_i)\cdot \abs{R}+s\cdot \abs{H}+j+i_0}(\bar{t}))\\&=&F_{f+g}(\varphi^{i_0}(z),\bar{t})= f(\varphi^{i_0}(z))=m_{\varphi,f}(z).
\end{eqnarray*}
This proves the affirmation and then  
$$m_{\varphi,f}(z)\in \mathcal{L}_{\Lambda\times \mathbb{S}^1,\varphi\times {\text{rot}}_{\alpha},F_{f+g}}\setminus \mathcal{M}_{\Lambda,\varphi,f_{F_{f+g}}}\subset\mathcal{L}_{\Lambda\times \mathbb{S}^1,\varphi\times {\text{rot}}_{\alpha},F_{f+g}}\setminus \mathcal{L}_{\Lambda,\varphi,f_{F_{f+g}}}.$$
As we wanted to see.
\end{proof}

\begin{remark}
If $HD(\Lambda)<1$, we will prove in Corollary \ref{where} that for generic $F\in C^r(S\times \mathbb{S}^1)$ it is also true that 
    $$\mathcal{L}_{\Lambda\times \mathbb{S}^1,\varphi\times {\text{rot}}_{\alpha},F}\neq\mathcal{L}_{\Lambda,\varphi,f_F} \ \ \text{and} \ \ \mathcal{M}_{\Lambda\times \mathbb{S}^1,\varphi\times {\text{rot}}_{\alpha},F}\neq\mathcal{M}_{\Lambda,\varphi,f_F}.$$   
\end{remark}

\subsection{The open and dense set of functions}

Fix $r\geq 2$. Given $F\in C^r(S\times \mathbb{S}^1,\mathbb{R})$, define the set 
$$M_{F}(\Lambda\times \mathbb{S}^1)=\{z\in\Lambda\times \mathbb{S}^1:F(z)\geq F(y) \ \ \forall \ y\in \Lambda\times \mathbb{S}^1\}.$$
This is the set of maximum points of $F|_{\Lambda\times \mathbb{S}^1}$. For $x\in \Lambda$, let $e^s_x$ and $e^u_x$ be unit vectors in the stable and unstable directions of $T_xS$.

\begin{proposition}
The set $\mathcal{R}_{\varphi,\Lambda}$ of the $F\in C^r(S\times \mathbb{S}^1,\mathbb{R})$ such that $\#M_{F}(\Lambda\times \mathbb{S}^1)=1$ and if $(\tilde{x},\tilde{t})\in \Lambda\times \mathbb{S}^1$ is the unique element of $M_{F}(\Lambda\times \mathbb{S}^1)$, then $DF_{(\tilde{x},\tilde{t})}(e^{s,u}_{\tilde{x}})\neq 0$ and $\frac{\partial^2F}{\partial t^2}(\tilde{x},\tilde{t})<0$, is open and dense.
\end{proposition}

Before proving this proposition we will recall some results. We say that $x$ is a boundary point of $\Lambda$ in the unstable direction, if $x$ is an accumulation point only from one side by points in $W^{u}_{loc}(x)\cap \Lambda$. If $x$ is a boundary point of $\Lambda$ in the unstable direction, then, due to the local product structure, the same holds for all points in $W^{s}(x)\cap \Lambda$. So the boundary points in the unstable direction are local intersections of local stable manifolds with $\Lambda$. For this reason we denote the set of boundary points in the unstable direction by $\partial_{s}\Lambda$. The boundary points in the stable direction are defined similarly. The set of these boundary points is denoted by $\partial_{u}\Lambda$.

It is well known that for any horseshoe $\Lambda$, as above, there are periodic points $p^{s}_{1},...,p^{s}_{n_{s}}$ and $p^{u}_{1},...,p^{u}_{n_{u}}$ with $n_s,n_u\in \mathbb{N}$  such that
$$\Lambda\cap \left(\bigcup_{i}W^{s}(p^{s}_{i})\right)=\partial_{s}\Lambda \ \ \text{and} \ \  \Lambda\cap \left(\bigcup_{i}W^{u}(p^{u}_{i})\right)=\partial_{u}\Lambda.$$
Moreover, $\partial_{s}\Lambda$, $\partial_{u}\Lambda$ and $\partial_{s}\Lambda \cap \partial_{u}\Lambda$ are dense in $\Lambda$. The set $\partial_{s}\Lambda \cap \partial_{u}\Lambda$ is the set of \textit{corners of rectangles of} $\Lambda$. 

\begin{lemma}
If $F\in C^r(S\times \mathbb{S}^1,\mathbb{R})$ and $(\tilde{x},\tilde{t})\in M_F(\Lambda\times \mathbb{S}^1)$ is such that $DF_{(\tilde{x},\tilde{t})}(e^{s,u}_{\tilde{x}})\neq 0$, then $\tilde{x}$ is a corner of rectangle of $\Lambda$.
\end{lemma}
\begin{proof}
Suppose $\tilde{x}=(\tilde{x}_s,\tilde{x}_u)$, where we are using the coordinates of the stable and unstable foliation. Observe that $\tilde{x}\in \partial_{s} \Lambda$ because $\tilde{x}_u$ is a point of maximum of the function $f:W^u_{loc}(\tilde{x})\cap \Lambda\rightarrow \mathbb{R}$ given by $f(x_u)=F(\tilde{x}_s,x_u,\tilde{t})$ and $\frac{df}{dx_u}(\tilde{x}_u)=DF_{(\tilde{x},\tilde{t})}(e^u_{\tilde{x}})\neq 0$, which let us conclude that $\tilde{x}$ cannot be accumulated from both sides in $W^u_{loc}(\tilde{x})\cap \Lambda$. As $DF_{(\tilde{x},\tilde{t})}(e^s_{\tilde{x}})\neq 0$, we can also conclude that $\tilde{x}\in \partial_{u} \Lambda$ and then, $\tilde{x}\in \partial_{s} \Lambda\cap \partial_{u}\Lambda$, as we wanted to see.
\end{proof}

In particular, for any $F\in R_{\varphi,\Lambda}$, if $(\tilde{x},\tilde{t})$ is the unique element of $M_F(\Lambda\times \mathbb{S}^1)$
one has that $\tilde{x}$ is a corner of rectangle of $\Lambda$.

\begin{lemma}\label{open}
The set $\mathcal{R}_{\varphi,\Lambda}\subset C^r(S\times \mathbb{S}^1,\mathbb{R})$ is open. 
\end{lemma}

\begin{proof}
Consider $F\in \mathcal{R}_{\varphi,\Lambda}$, a rectangle $R$ and an open arc $I\subset \mathbb{S}^1$ such that $(\tilde{x},\tilde{t})\in R\times I$, where $M_F(\Lambda\times \mathbb{S}^1)=\{(\tilde{x},\tilde{t})\}$,  $\tilde{x}$ is a corner of $R$, the derivatives $DF_{(x,t)}(e^{s,u}_{x})$, $\frac{\partial^2F}{\partial t^2}(x,t)$ are far from $0$ and have the same signs of $DF_{(\tilde{x},\tilde{t})}(e^{s,u}_{\tilde{x}})$, $\frac{\partial^2F}{\partial t^2}(\tilde{x},\tilde{t})$ for $(x,t)\in R\times I$. Then, one has that $(\Lambda\times\mathbb{S}^1)\setminus (R\times I)$ is a compact set and then $\max F|_{(\Lambda\times\mathbb{S}^1)\setminus (R\times I)}\\<F(\tilde{x},\tilde{t})$. 

If $G$ is $C^r$-close enough to $F$, each point of maximum of
$G|_{\Lambda\times \mathbb{S}^1}$ belongs to $R\times I$ and $DG_{(x,t)}(e^{s,u}_{x})$ and $\frac{\partial^2G}{\partial t^2}(x,t)$ have the same signs of $DF_{(\tilde{x},\tilde{t})}(e^{s,u}_{\tilde{x}})$ and $\frac{\partial^2F}{\partial t^2}(\tilde{x},\tilde{t})$ for $(x,t)\in R\times I$. In particular, if $(\bar{x},\bar{t})\in M_G(\Lambda\times \mathbb{S}^1)$, as $\bar{g}(x)=G(x,\bar{t})$ has a point of maximum in $\bar{x}$ and $D\bar{g}_{\bar{x}} (e^s_{\bar{x}})=DG_{(\bar{x},\bar{t})}(e^s_{\bar{x}})$ and $D\bar{g}_{\bar{x}} (e^u_{\bar{x}})=DG_{(\bar{x},\bar{t})}(e^u_{\bar{x}})$, we conclude that $\bar{x}=\tilde{x}$. On the other hand $\bar{t}$ is also determined with uniqueness since if $(\tilde{x},\bar{t}_1),(\tilde{x},\bar{t}_2)\in M_G(\Lambda\times \mathbb{S}^1)$ are different, then for the mean value theorem, for some $t_3\in I$ one has
$$\frac{\partial^2G}{\partial t^2}(\tilde{x},t_3)(\bar{t}_2-\bar{t}_1)=\frac{\partial G}{\partial t}(\tilde{x},\bar{t}_2)-\frac{\partial G}{\partial t}(\tilde{x},\bar{t}_1)=0,$$
which is a contradiction. Then $\#M_G(\Lambda\times \mathbb{S}^1)=1$, that concludes the proof of the lemma.
  
\end{proof}

\begin{lemma}\label{denso}
The set $\mathcal{R}_{\varphi,\Lambda}$ is dense in $C^r(S\times \mathbb{S}^1,\mathbb{R})$.
\end{lemma}

\begin{proof}
Remember that the set $M$ of Morse functions is dense in $C^r(S\times \mathbb{S}^1,\mathbb{R})$ and that if $F$ is a Morse function, then $Crit(F)=\left\{z\in S\times \mathbb{S}^1: Df_{z}=0\right\}$ is a discrete set. In particular, since $S\times \mathbb{S}^1$ is a compact set, we have $\# Crit(F)<\infty$ and as $int (\Lambda\times \mathbb{S}^1)=\emptyset$, we can find some small perturbation $F_1\in C^r(S\times \mathbb{S}^1,\mathbb{R})$ of $F$ such that $(\Lambda\times \mathbb{S}^1)\cap Crit(F_1)=\emptyset$. Fix $\epsilon>0$, some point $(\tilde{x},\tilde{t})\in M_{F_1}(\Lambda\times \mathbb{S}^1)$ and consider a non-negative function $F_{\epsilon}$, $\epsilon$-close to the constant function $0$ and such that it has in $(\tilde{x},\tilde{t})$ a unique point of maximum. If $\epsilon$ is small enough, the function $F_2=F_1+F_{\epsilon}$ is a small perturbation of $F$ that satisfies $(\Lambda\times \mathbb{S}^1)\cap Crit(F_2)=\emptyset$ and $M_{F_2}(\Lambda\times \mathbb{S}^1)=\{(\tilde{x},\tilde{t})\}$. Note that if we consider a non-positive small perturbation, $F_a\in C^r(\mathbb{S}^1,\mathbb{R})$, of the function $0$ that close to $\tilde{t}$ takes the form $-a(t-\tilde{t})^2$ where $a>0$ is small then, for $F_3=F_2+F_a$, one has $M_{F_3}(\Lambda\times \mathbb{S}^1)=\{(\tilde{x},\tilde{t})\}$, $(\Lambda\times \mathbb{S}^1)\cap Crit(F_3)=\emptyset$ and $\frac{\partial^2 F_3}{\partial t^2}(\tilde{x},\tilde{t})<0$. Arguing as in the proof of Proposition \ref{differentiable} and Lemma \ref{open}, one concludes that, for every $x$ close to $\tilde{x}$, the map $g_x(t)=F_3(x,t)$ has a unique point of maximum $\tilde{t}(x)$ close to $\tilde{t}$, and that $\tilde{t}(x)$ is a $C^{r-1}$ function of $x$ defined in a neighborhood $W\subset S$ of $\tilde{x}$. 

As the function $g:\mathbb{S}^1\rightarrow \mathbb{R}$ given by $g(t)=F_3(\tilde{x},t)$ has a point of maximum in $\tilde{t}$, one has that $\frac{\partial F_3}{\partial t}(\tilde{x},\tilde{t})=0$ and then we conclude that either $D{F_3}_{(\tilde{x},\tilde{t})}(e^s_{\tilde{x}})\neq 0$ or $D{F_3}_{(\tilde{x},\tilde{t})}(e^u_{\tilde{x}})\neq 0$. If both $D{F_3}_{(\tilde{x},\tilde{t})}(e^s_{\tilde{x}})$ and $D{F_3}_{(\tilde{x},\tilde{t})}(e^u_{\tilde{x}})$ are non-zero, there is nothing to do. Otherwise, consider the function $f:W\rightarrow \mathbb{R}$ given by $f(x)=F_3(x,\tilde{t}(x))$ (note that $Df_x (e^{s,u}_x)=D{F_3}_{(x,\tilde{t}(x))}(e^{s,u}_x)$) and suppose that $Df_{\tilde{x}} (e^s_{\tilde{x}})=D{F_3}_{(\tilde{x},\tilde{t})}(e^s_{\tilde{x}})= 0$ and $Df_{\tilde{x}} (e^u_{\tilde{x}})=D{F_3}_{(\tilde{x},\tilde{t})}(e^u_{\tilde{x}})\neq 0$.

Then, there is a $C^r$-neighborhood $\mathcal{V}$ of $f$ (in $C^r(S,\mathbb{R})$) and a neighborhood $U$ of $\tilde{x}$, such that, if $x\in U\cap \Lambda$ and $\tilde{f}\in \mathcal{V}$, then $D\tilde{f}_x(e_x^u)\neq 0$. Let $\mathcal{R}$ be a Markov partition of $\Lambda$ such that the element $R_{\tilde{x}}$ of $\mathcal{R}$ containing $\tilde{x}$ is contained in $U$. Without loss of generality, we can assume that $U$ is contained in a local chart $\varphi:\tilde{U}\subset M\rightarrow V\subset \mathbb{R}^2$ with $U\subset \tilde{U}$ and $\tilde{U}\cap \tilde{R}=\emptyset$ for all $\tilde{R}\in \mathcal{R}\setminus \{R_{\tilde{x}}\}$. Observe that since $D\tilde{f}_{\tilde{x}}(e^u_{\tilde{x}})\neq 0$, $\tilde{x}\in \partial_s \Lambda$. Therefore the possible maximum points of $\tilde{f}$ in $\Lambda\cap R_{\tilde{x}}$ are on $W^s_{loc}(\tilde{x})\cap \Lambda$, which has zero Lebesgue measure. Consider the function $\theta^s:(W^s_{loc}(\tilde{x})\cap \Lambda)\times \mathbb{R}\rightarrow \mathbb{R}^2$ defined by 
$$\theta^s(x,y)=\nabla(f\circ\varphi^{-1})(\varphi(x))-y \begin{pmatrix}
0 & -1  \\
1 & 0 \\
\end{pmatrix}
D\varphi_x(e^s_x),$$
where the above matrix is the orthogonal rotation. Since $\theta^s$ extends to a $C^1$-function, then the Lebesgue measure of $\theta^s((W^s_{loc}(\tilde{x})\cap \Lambda)\times \mathbb{R})$ is zero. Therefore, there is some  $v\in\mathbb{R}^2$ with norm arbitrarily small such that $v\notin \theta^s((W^s_{loc}(\tilde{x})\cap \Lambda)\times \mathbb{R})$. Put $h(z)=f\circ \varphi^{-1}(z)-\langle v,z\rangle$ for $z\in V$. Thus $D(h\circ \varphi)_xe^s_x=Dh_{\varphi(x)}D\varphi_x(e^s_x)\neq 0$ for all $x\in W^s_{loc}(\tilde{x})\cap \Lambda$. 

Consider then the perturbation $F_4$ of $F_3$ given for $(x,t)\in \tilde{U}\times \mathbb{S}^1$ by $F_4(x,t)=F_3(x,t)-\langle v,\varphi(x)\rangle$ and take $(\bar{x},\bar{t})\in M_{F_4}(\Lambda\times \mathbb{S}^1)$. Then $(\bar{x},\bar{t})$ is close to $(\tilde{x},\tilde{t})$, and thus $\bar{x}\in W$ and $\bar{t}=\tilde{t}(\bar{x})$. By construction, we have that both $D{F_4}_{(\tilde{x},\tilde{t})}(e^s_{\tilde{x}})$ and $D{F_4}_{(\tilde{x},\tilde{t})}(e^u_{\tilde{x}})$ are non-zero, and we are done.

\end{proof}

\subsection{Interior of the Lagrange spectrum}

Fix a diffeomorphism $\varphi$ of a surface $S$ with a horseshoe $\Lambda$ and associate with each $x\in S$ an orientation preserving diffeomorphism $R_x:\mathbb{S}^1\rightarrow \mathbb{S}^1$. Now, consider the set $L_{\Lambda\times \mathbb{S}^1,\Phi,F}$ where $\Phi:S\times \mathbb{S}^1\rightarrow S\times \mathbb{S}^1$ given by $\Phi(x,t)=(\varphi(x),R_x(t))$ is supposed to be a diffeomorphism (with inverse given by $\Phi^{-1}(x,t)=(\varphi^{-1}(x),R^{-1}_{\varphi^{-1}(x)}(t))$) and $F:S\times \mathbb{S}^1\rightarrow \mathbb{R}$ is differentiable. To simplify, given a family of diffeomorphisms $\{R_s \}_{s=0}^n$ of $\mathbb{S}^1$ where $n\in \mathbb{N}$, define $\bigcirc_{s=0}^n R_s:=  
R_n \circ \dots \circ R_1\circ R_0$. Thus with this notation, given $(x,t)\in S\times \mathbb{S}^1$ one has $\Phi^n(x,t)=(\varphi^n(x),\bigcirc_{s=0}^{n-1} R_{\varphi^s(x)}(t))$. We start with the following lemma

\begin{lemma}\label{pertur}
Suppose one has two sequences $\{R_n\}_{n\in \mathbb{N}}$ and $\{\tilde{R}_n\}_{n\in \mathbb{N}}$ of real functions such that each $R_n$ is differentiable and there is a constant $C>0$ with $\max_{x\in \mathbb{R}}\abs{R^{'}_n(x)}\leq C$ for all $n\in \mathbb{N}$. Given $n\in \mathbb{N}$ define $G_n=\bigcirc_{\tilde{s}=1}^{n} R_{\tilde{s}}$ and $\tilde{G}_n=\bigcirc_{\tilde{s}=1}^{n} \tilde{R}_{\tilde{s}}$. Then, if each $r_n=\tilde{R}_n-R_n$ is bounded, one has 
$$\tilde{G}_n=G_n+\Sigma_{k=1}^n \tilde{r}^n_k$$
where for $k=1,\dots,n$
$$\max_{x\in \mathbb{R}}\abs{\tilde{r}^n_k(x)}\leq C^{n-k}\cdot \max_{x\in \mathbb{R}}\abs{r_k(x)}.$$
\end{lemma}

\begin{proof}
The proof is by induction: if $n=1$, we can take $\tilde{r}^1_1=r_1$. Now, suppose that $\tilde{G}_n=G_n+\Sigma_{k=1}^n \tilde{r}^n_k$ where 
for $k=1,\dots,n$ one has $\max_{x\in \mathbb{R}}\abs{\tilde{r}^n_k(x)}\leq C^{n-k}\cdot \max_{x\in \mathbb{R}}\abs{r_k(x)}.$ Given $x\in \mathbb{R}$ one has by the mean value theorem
\begin{eqnarray*}
\tilde{G}_{n+1}(x)&=&\tilde{R}_{n+1}(\tilde{G}_n(x))=\tilde{R}_{n+1}(G_n(x) +\Sigma_{k=1}^n \tilde{r}^n_k(x))\\ &=& R_{n+1}(G_n(x) +\Sigma_{k=1}^n \tilde{r}^n_k(x))+r_{n+1}(G_n(x) +\Sigma_{k=1}^n \tilde{r}^n_k(x))\\&=&R_{n+1}(G_n(x))+R^{'}_{n+1}(\eta_{x,n})\cdot(\Sigma_{k=1}^n \tilde{r}^n_k(x))+ r_{n+1}(G_n(x) +\Sigma_{k=1}^n \tilde{r}^n_k(x))\\ &=& G_{n+1}(x)+ \Sigma_{k=1}^{n+1} \tilde{r}^{n+1}_k(x),
\end{eqnarray*}
where $\eta_{x,n}$ is between $G_n(x)$ and $G_n(x)+\Sigma_{k=1}^n \tilde{r}^n_k(x)$. Thus
$$\tilde{G}_{n+1}(x)=G_{n+1}(x)+ \Sigma_{k=1}^{n+1} \tilde{r}^{n+1}_k(x),$$
where $\tilde{r}^{n+1}_k(x)=R^{'}_{n+1}(\eta_{x,n})\cdot \tilde{r}^n_k(x)$ for $k=1,\dots n$ and $\tilde{r}^{n+1}_{n+1}(x)=r_{n+1}(G_n(x) +\Sigma_{k=1}^n \tilde{r}^n_k(x))$, and then 
$$\max_{x\in \mathbb{R}}\abs{\tilde{r}^{n+1}_k(x)}\leq C\cdot \max_{x\in \mathbb{R}}\abs{\tilde{r}^{n}_k(x)} \leq C^{n+1-k}\cdot \max_{x\in \mathbb{R}}\abs{r_k(x)} $$
if $k=1,\dots n$ and 
$$\max_{x\in \mathbb{R}}\abs{\tilde{r}^{n+1}_{n+1}(x)}\leq \max_{x\in \mathbb{R}}\abs{r^{n+1}(x)}=C^0\cdot \max_{x\in \mathbb{R}}\abs{r^{n+1}(x)} .$$
This finishes the proof of the lemma.
\end{proof}

\begin{lemma}\label{lemarota}
 Let $q\in S$ be a periodic point of $\varphi$ of period $k\in \mathbb{N}$, such that $q=\Pi^{-1}(\bar{Q})$ for some admissible word 
 $Q$ and $\psi$ a diffeomorphism of $\mathbb{S}^1$. Suppose that $\bigcirc_{s=0}^{k-1} R_{\varphi^s(q)}$ has irrational rotation number and the conjugacy given by Denjoy's theorem is a diffeomorphism. Suppose $H_1$ and  $H_2$ are infinite sequences on the left and on the right respectively such that $H_1QH_2$ is also admissible. Then, given $\epsilon>0$, $r\in \mathbb{N}$ and $t_1,t_2\in \mathbb{S}^1$ there are arbitrary large numbers $N\in\mathbb{N}$ such that if $p_N=\Pi^{-1}(H_1;Q^NH_2)$ then one has
$$d(\psi\circ \bigcirc_{s=0}^{N\cdot k-1+r} R_{\varphi^s(p_N)}(t_1),t_2)<\epsilon.$$
\end{lemma}
\begin{proof}

Let $g:\mathbb{S}^1\rightarrow \mathbb{S}^1$ be the diffeomorphism given by Denjoy's theorem   such that $g\circ\bigcirc_{s=0}^{k-1} R_{\varphi^s(q)} \circ g^{-1}={\text{rot}}_{\alpha}$. Additionally, note that the function $P:S^k\times \mathbb{S}^1\rightarrow \mathbb{S}^1$
given by $P(x_0,\dots, x_{k-1},t)=(g \circ \bigcirc_{s=0}^{k-1} R_{x_s}\circ g^{-1})(t)$ is differentiable, then we can find a constant $\tilde{C}_1>0$ such that for $x_0,\dots,x_{k-1}, \tilde{x}_0,\dots,\tilde{x}_{k-1}\in S$ one has
\begin{equation}
\max_{t\in \mathbb{S}^1}d(P(x_0,\dots, x_{k-1},t),P(\tilde{x}_0,\dots, \tilde{x}_{k-1},t))\leq \tilde{C}_1\cdot \Sigma_{s=0}^{k-1}d(x_s,\tilde{x}_s).
\end{equation}

In particular, if $N_1, N\in \mathbb{N}$ are such that $N>2N_1$,     $\tilde{s}\in \{N_1,\dots, N-N_1-1\}$ and  $x_0=\varphi^{\tilde{s}\cdot k}(p_N),x_1=\varphi^{\tilde{s}\cdot k+1}(p_N),\dots, x_{k-1}=\varphi^{\tilde{s}\cdot k+k-1}(p_N)$, $\tilde{x}_0=\varphi^{\tilde{s}\cdot k}(q)=q,\tilde{x}_1=\varphi^{\tilde{s}\cdot k+1}(q)=\varphi(q),\dots, \tilde{x}_{k-1}=\varphi^{\tilde{s}\cdot k+k-1}(q)=\varphi^{k-1}(q)$ one has
\begin{equation}\label{2.2}
\begin{aligned}
 & \hspace{0.5cm}   \max_{t\in \mathbb{S}^1}d((g \circ \bigcirc_{s=0}^{k-1} R_{\varphi^{\tilde{s}\cdot k+s}(p_N)}\circ g^{-1})(t),{\text{rot}}_{\alpha}(t)) \\
 &= \max_{t\in \mathbb{S}^1}d((g \circ \bigcirc_{s=0}^{k-1} R_{\varphi^{\tilde{s}\cdot k+s}(p_N)}\circ g^{-1})(t),(g \circ \bigcirc_{s=0}^{k-1} R_{\varphi^s(q)}\circ g^{-1})(t)) \\
 &\leq \tilde{C}_1\cdot \Sigma_{s=0}^{k-1}d(\varphi^{\tilde{s}\cdot k+s}(p_N),\varphi^s(q)). 
\end{aligned}
\end{equation}

Additionally, we have for some constant $\tilde{C}_2>0$ 

\begin{equation}\label{2.3}
\begin{aligned}
 & \hspace{0.5cm}   \max_{t\in \mathbb{S}^1}d((\bigcirc_{s=kN_1}^{N\cdot k-k\cdot N_1-1} R_{\varphi^s(p_N)})(t),(\bigcirc_{s=0}^{k-1} R_{\varphi^s(q)})^{N-2N_1}(t)) \\
 &= \max_{t\in \mathbb{S}^1}d((g^{-1}\circ g\circ \bigcirc_{s=kN_1}^{N\cdot k-k\cdot N_1-1} R_{\varphi^s(p_N)}\circ g^{-1})(g(t)),(g^{-1}\circ {\text{rot}}_{\alpha}^{N-2N_1})(g(t))) \\
 &\leq \tilde{C}_2\cdot \max_{t\in \mathbb{S}^1}d((g \circ \bigcirc_{s=kN_1}^{N\cdot k-k\cdot N_1-1} R_{\varphi^s(p_N)}\circ g^{-1})(t),({\text{rot}}_{\alpha}^{N-2N_1})(t))\\ 
 &= \tilde{C}_2\cdot \max_{t\in \mathbb{S}^1}d((\bigcirc_{\tilde{s}=N_1}^{N-N_1-1}  g\circ \bigcirc_{s=0}^{k-1}R_{\varphi^{\tilde{s}\cdot k+s}(q)}\circ g^{-1})(t),({\text{rot}}_{\alpha}^{N-2N_1})(t)). 
\end{aligned}
\end{equation}

By (\ref{2.2}), (\ref{2.3}) and Lemma \ref{pertur} applied to some lifts $R_{\tilde{s}}$ of ${\text{rot}}_{\alpha}$ (that has derivative equal to $1$) and $\tilde{R}_{\tilde{s}}$ of $g \circ \bigcirc_{s=0}^{k-1}R_{\varphi^{\tilde{s}\cdot k+s}(p_N)}\circ g^{-1}$, one concludes after projecting again to $\mathbb{S}^1$ that 
\begin{equation}\label{2.4}
\begin{aligned}
 & \hspace{0.5cm}   \max_{t\in \mathbb{S}^1}d((\bigcirc_{s=kN_1}^{N\cdot k-k\cdot N_1-1} R_{\varphi^s(p_N)})(t),(\bigcirc_{s=0}^{k-1} R_{\varphi^s(q)})^{N-2N_1}(t)) \\
 &\leq \tilde{C}_2\cdot  \Sigma_{\tilde{s}= N_1}^{N-N_1-1} \max_{t\in \mathbb{S}^1}d((g \circ \bigcirc_{s=0}^{k-1} R_{\varphi^{\tilde{s}\cdot k+s}(p_N)}\circ g^{-1})(t),{\text{rot}}_{\alpha}(t)) \\
 &\leq \tilde{C}_1\cdot \tilde{C}_2\cdot\Sigma_{\tilde{s}= N_1}^{N-N_1-1} \Sigma_{s=0}^{k-1}d(\varphi^{\tilde{s}\cdot k+s}(p_N),\varphi^s(q))\\ 
 &= \tilde{C}_1\cdot \tilde{C}_2\cdot\Sigma_{s=kN_1}^{N\cdot k-k\cdot N_1-1} d(\varphi^s(p_N),\varphi^{res(s)}(q)). 
\end{aligned}
\end{equation}
where $res(s)$ denotes the remainder of the division of $s$ by $k$.

Given $\rho>0$, take $\eta>0$ such that 
$$2C\cdot \tilde{C}_1\cdot \tilde{C}_2 \cdot\eta^c \cdot \Sigma_{s=0}^{\infinity} 2^{-sc}<\rho,$$ 
where $c$ and $C$ are such that $ d(\Pi^{-1}(x),\Pi^{-1}(x))\leq C\cdot d(x,y)^c$ for every $x,y\in \Sigma_{\mathcal{B}}$, and consider $N$ even and $N_1\in \mathbb{N}$ large enough such that   
$$d(\sigma^{k\cdot N_1}(H_1;Q^NH_2),\bar{Q})\leq 2^{-k\cdot N_1}<\eta.$$ 
Then for $0\leq s < N k/2-kN_1$, one has $$d(\sigma^{k\cdot N_1+s}(H_1;Q^NH_2),\sigma^{res(s)}(\bar{Q}))\leq 2^{-(k\cdot N_1+s)}=2^{-k\cdot N_1}   2^{-s}$$ 
and  
$$d(\sigma^{Nk/2+s}(H_1;Q^NH_2),\sigma^{res(s)}(\bar{Q}))\leq 2^{-(Nk/2-s)}=2^{-k\cdot N_1}   2^{-(Nk/2-kN_1-s)}.$$
So, as $\Pi$ conjugates $\varphi$ and $\sigma$, we have 
\begin{eqnarray*}
  d(\varphi^{k\cdot N_1+s}(p_N),\varphi^{res(s)}(q))&=& d(\Pi^{-1}(\sigma ^{k\cdot N_1+s}(H_1;Q^NH_2)),\Pi^{-1}(\sigma^{res(s)}(\bar{Q})))\\ &\leq& C \cdot (d(\sigma^{k\cdot N_1+s}(H_1;Q^NH_2),\sigma^{res(s)}(\bar{Q}))^c\\ &\leq& C\eta^c 2^{-sc}    
\end{eqnarray*}
and 
\begin{eqnarray*}
  d(\varphi^{Nk/2+s}(p_N),\varphi^{res(s)}(q))&=& d(\Pi^{-1}(\sigma ^{Nk/2+s}(H_1;Q^NH_2)),\Pi^{-1}(\sigma^{res(s)}(\bar{Q})))\\ &\leq& C \cdot (d(\sigma^{Nk/2+s}(H_1;Q^NH_2),\sigma^{res(s)}(\bar{Q}))^c\\ &\leq& C\eta^c 2^{-(Nk/2-kN_1-s)c}.    
\end{eqnarray*}  
Now, observe that
$$\Sigma_{s=0}^{N k/2-kN_1} C\eta^c(2^{-sc}+2^{-(Nk/2-kN_1-s)c})=2C\eta^c \Sigma_{s=0}^{N k/2-kN_1} 2^{-sc}\leq 2C\eta^{c} \Sigma_{s=0}^{\infinity} 2^{-sc}$$
and then, in these conditions, by (\ref{2.4})
\begin{eqnarray*}
&&\max_{t\in \mathbb{S}^1}d((\bigcirc_{s=kN_1}^{N\cdot k-k\cdot N_1-1} R_{\varphi^s(p_N)})(t),(\bigcirc_{s=0}^{k-1} R_{\varphi^s(q)})^{N-2N_1}(t))\\ &\leq & \tilde{C}_1\cdot \tilde{C}_2\cdot\Sigma_{s=kN_1}^{N\cdot k-k\cdot N_1-1} d(\varphi^s(p_N),\varphi^{res(s)}(q))<\rho.
\end{eqnarray*}
As before, for any $n\in\mathbb{N}$, the function $(x_0,\dots, x_n,t)\rightarrow \bigcirc_{s=0}^{n} R_{x_s}(t)$ is differentiable. As we have the limits 
$$\lim _{N\rightarrow \infinity}(p_N,\dots,\varphi^{k\cdot N_1-1}(p_N))=(\Pi^{-1}(H_1;Q^{\infinity}), \dots, \varphi^{k\cdot N_1-1}(\Pi^{-1}(H_1;Q^{\infinity})))$$
and 
$$\lim _{N\rightarrow \infinity}(\varphi^{kN-kN_1}(p_N),\dots,\varphi^{N\cdot k-1+r}(p_N))=(\varphi^{-k\cdot N_1}(\Pi^{-1}(Q^{\infinity};H_2)), \dots\varphi^{r-1}(\Pi^{-1}(Q^{\infinity};H_2)))$$
then, for $N$ large, $\bigcirc_{s=0}^{kN_1-1} R_{\varphi^s(p_N)}$ and $\psi\circ \bigcirc_{s=kN-kN_1}^{N\cdot k-1+r} R_{\varphi^s(p_N)}$ are uniformly close to \\ $\bigcirc_{s=0}^{kN_1-1} R_{\varphi^s(\Pi^{-1}(H_1;Q^{\infinity}))}$ and $\psi\circ  \bigcirc_{s=-kN_1}^{r-1} R_{\varphi^s(\Pi^{-1}(Q^{\infinity};H_2))}$ respectively. Thus, as 
$$\psi\circ\bigcirc_{s=0}^{N\cdot k-1+r} R_{\varphi^s(p_N)}=\psi\circ\bigcirc_{s=kN-kN_1}^{N\cdot k-1+r} R_{\varphi^s(p_N)}\circ \bigcirc_{s=kN_1}^{N\cdot k-kN_1-1} R_{\varphi^s(p_N)}\circ\bigcirc_{s=0}^{kN_1-1} R_{\varphi^s(p_N)}$$
and 
$\bigcirc_{s=0}^{k-1}R_{ \varphi^s(q)}$ has dense orbits, we can find $N$ arbitrarily large such that 
$$d(\psi\circ \bigcirc_{s=0}^{N\cdot k-1+r} R_{\varphi^s(p_N)}(t_1),t_2)<\epsilon.$$
As we wanted to see.
\end{proof}

\begin{remark}
An irrational number $\alpha$ is said to belong to the Diophantine class $D_{\delta}$ ($\delta\geq 0$) if there exists a constant $K>0$ such that $\abs{\alpha-p/q}\geq K\cdot q^{-2-\delta}$ for any rational number $p/q$. The class $D_{\delta}$ has Lebesgue measure $1$. In \cite{her} is proved the following result: Let $T$ be a $C^{2+\epsilon}$-smooth orientation-preserving diffeomorphism of $\mathbb{S}^1$ with rotation number $\alpha\in D_{\delta}$, $0\leq \delta<\epsilon\leq 1$, $\epsilon-\delta<1$. Then the conjugacy, given by the Denjoy's Theorem for $T$ is of class $C^{1+\epsilon-\delta}$. There are many other related results that imply that the conjugacy given by the Denjoy's Theorem is a diffeomorphism, see for example \cite{Y1} and \cite{Y2}.

\end{remark}

\begin{proposition}\label{rphi1}
Given $F\in \mathcal{R}_{\varphi,\Lambda}$, if we can find a periodic point $q\in \Lambda$ of $\varphi$ with period $k\in\mathbb{N}$ such that $\tilde{x}\notin \mathcal{O}(q)$, where $M_F(\Lambda\times \mathbb{S}^1)=\{(\tilde{x},\tilde{t})\}$, the rotation number of $\bigcirc_{s=0}^{k-1} R_{\varphi^s(q)}$ is irrational and the conjugacy function given by Denjoy's theorem is a diffeomorphism, then 
there is a sequence of non-trivial intervals contained in $\mathcal{L}_{\Lambda\times \mathbb{S}^1,\Phi,F}$ converging to the maximum $F(\tilde{x},\tilde{t})$.
\end{proposition}

\begin{proof}

Consider an admissible word $Q$ such that $q=\Pi^{-1}(\bar{Q})$. Given $\delta>0$, there exists $\epsilon_{\delta}>0$ such that if $d(x,\tilde{x})\geq \delta$ or $\abs{t-\tilde{t}}\geq \delta$, then $F(x,t)<F(\tilde{x},\tilde{t})-\epsilon_{\delta}$. Take $0<\delta<d(\mathcal{O}(q),\tilde{x})$, $I$ a small enough closed interval centered at $\tilde{t}$ and $z\in W^s(q)\cap W^u(q)$ close of $\tilde{x}$ such that for $t\in I$, $F(z,t)>F(\tilde{x},\tilde{t})-\epsilon_{\delta}$. Suppose, additionally, that $\delta$ is so small such that if $(x,t)\in B(\tilde{x},\delta)\times (\tilde{t}-\delta,\tilde{t}+\delta)$, $\frac{\partial^2 F(x,t)}{\partial t^2}<0$.  

If $j> 0$ is such that  $\Pi(\varphi^{-j}(z))=Q^{\infinity};HQ^{\infinity}$, where $QHQ$ is admissible and $H$ is large such that in the positions that are not in $H$ the point determined does not belong to $B(\tilde{x},\delta)$ ($H$ can start or finish with some powers of $Q$), then for $t\in (\bigcirc_{s=0}^{j-1} R_{\varphi^{s-j}(z)})^{-1}(I)$ one has 
$$\Phi^j(\varphi^{-j}(z),t)=(z,\bigcirc_{s=0}^{j-1} R_{\varphi^{s-j}(z)}(t))\in \{z\}\times I.$$
Then if $H=h_0\dots h_n$, for any $t\in (\bigcirc_{s=0}^{j-1} R_{\varphi^{s-j}(z)})^{-1}(I)$, there exists $i\in \{1,\dots, n\}$, 
such that $F(\Phi^i(\varphi^{-j}(z),t))\geq  
F(\Phi^l(\varphi^{-j}(z),t))$, $\forall l\in \mathbb{Z}$. Given $r\in \{1,\dots, n\}$ let 
$$\mathcal{A}_r=\{t\in (\bigcirc_{s=0}^{j-1} R_{\varphi^{s-j}(z)})^{-1}(I):F(\Phi^r(\varphi^{-j}(z),t))\geq  
F(\Phi^l(\varphi^{-j}(z),t)),\ \forall l\in \mathbb{Z}\}.$$

The sets $\mathcal{A}_r$ are closed and then, there exists $j_0\in \{1,\dots, n\}$ such that $\mathcal{A}_{j_0}$ has non-empty interior (Baire's Theorem). Note that, for $t\in \mathcal{A}_{j_0}$, $$(\varphi^{j_0-j}(z),\bigcirc_{s=0}^{j_0-1} R_{\varphi^{s-j}(z)}(t))\in B(\tilde{x},\delta)\times (\tilde{t}-\delta,\tilde{t}+\delta)$$
and 
$$\frac{dF}{dt} (\varphi^{j_0-j}(z),\bigcirc_{s=0}^{j_0-1} R_{\varphi^{s-j}(q)}(t))=\frac{\partial F}{\partial t} (\varphi^{j_0-j}(z),\bigcirc_{s=0}^{j_0-1} R_{\varphi^{s-j}(q)}(t))\cdot (\bigcirc_{s=0}^{j_0-1} R_{\varphi^{s-j}(z)})^{'}(t).$$
Then, $F(\varphi^{j_0-j}(z),\bigcirc_{s=0}^{j_0-1} R_{\varphi^{s-j}(z)}(\mathcal{A}_{j_0}))$ has non-empty interior. We will see that 
$$F(\varphi^{j_0-j}(z),\bigcirc_{s=0}^{j_0-1} R_{\varphi^{s-j}(z)}(\mathcal{A}_{j_0}))\subset \mathcal{L}_{\Lambda\times \mathbb{S}^1,\Phi,F}.$$
Given $t\in \mathcal{A}_{j_0}$, by Lemma \ref{lemarota}, we can define the point 
$$z_t=\Pi^{-1}(Q^{\infinity};HQ^{n_1}HQ^{n_2}H\dots Q^{n_{k-1}}HQ^{n_k}\dots ),$$ 
where the sequence $\{n_s\}_{s\in\mathbb{N}}$ is taken in such a way that:
\begin{itemize}
    \item $n_s\rightarrow\infinity$
    \item $\bigcirc_{\tilde{s}=0}^{(\Sigma_{i=1}^{s}n_i)\cdot \abs{Q}+s\cdot \abs{H}}R_{\varphi^{\tilde{s}}(z_t)}(t)\rightarrow t.$
\end{itemize}    
We have that
$$\ell_{\Phi,F}(z_t,t)= \limsup_{n\rightarrow\infinity} F(\varphi^n(z_t),\bigcirc_{\tilde{s}=0}^{n-1} R_{\varphi^{\tilde{s}}(z_t)}(t))=F(\varphi^{j_0-j}(z),\bigcirc_{s=0}^{j_0-1} R_{\varphi^{s-j}(z)}(t)).$$
Indeed, for $n_s$ big, in the positions corresponding to the blocks $Q$, the value of $F$ is small because $\delta<d(\mathcal{O}(q),\tilde{x})$. If $l\in \{1,\dots, n\}$ one has 
\begin{eqnarray*}
&&\lim _{s\rightarrow\infinity}\Phi^{(\Sigma_{i=1}^{s}n_i)\abs{Q}+s\abs{H}+l}(z_t,t)\\&=&\lim _{s\rightarrow\infinity}(\varphi^{(\Sigma_{i=1}^{s}n_i) \abs{Q}+s \abs{H}+l}(z_t),\bigcirc_{\tilde{s}=0}^{(\Sigma_{i=1}^{s}n_i) \abs{Q}+s \abs{H}+l-1} R_{\varphi^{\tilde{s}}(z_t)}(t))\\&=&\lim _{s\rightarrow\infinity}(\varphi^{(\Sigma_{i=1}^{s}n_i) \abs{Q}+s \abs{H}+l}(z_t),\bigcirc_{\tilde{s}=0}^{l-1}R_{\varphi^{\tilde{s}+(\Sigma_{i=1}^{s}n_i) \abs{Q}+s \abs{H}}(z_t)}\circ\bigcirc_{\tilde{s}=0}^{(\Sigma_{i=1}^{s}n_i) \abs{Q}+s \abs{H}} R_{\varphi^{\tilde{s}}(z_t)}(t))\\&=&(\varphi^{l-j}(z),\bigcirc_{\tilde{s}=0}^{l-1} R_{\varphi^{\tilde{s}-j}(z)}(t))
\end{eqnarray*}
and then, as $t\in \mathcal{A}_{j_0}$, we have that the maximum value of $F$ is attained if $l=j_0$, that shows the result.

Finally, as we can take $z$ arbitrarily close to $\tilde{x}$, then there is a sequence of intervals contained in $\mathcal{L}_{\Lambda\times \mathbb{S}^1,\Phi,F}$ converging to $F(\tilde{x},\tilde{t})$.
\end{proof}

\begin{proposition}\label{rphi2}
Let $F\in \mathcal{R}_{\varphi,\Lambda}$ and suppose $\tilde{x}$ is not periodic, where $M_F(\Lambda\times \mathbb{S}^1)=\{(\tilde{x},\tilde{t})\}$. If $\tilde{x}\in W^s(q_2)\cap W^u(q_1)$ where $q_1, q_2\in \Lambda$ are periodic points of $\varphi$ with periods $k_1$ and $k_2$, respectively, and either the rotation number of $\bigcirc_{s=0}^{k_1-1} R_{\varphi^s(q_1)}$ or the rotation number of $\bigcirc_{s=0}^{k_2-1} R_{\varphi^s(q_2)}$ is irrational with the conjugacy function given by Denjoy's theorem a diffeomorphism, we can find a non-trivial interval containing $F(\tilde{x},\tilde{t})$ that is contained in $\mathcal{L}_{\Lambda\times \mathbb{S}^1,\Phi,F}$.
\end{proposition}
\begin{proof}

In this case, given $\delta>0$ there exists $\epsilon_{\delta}>0$ such that $d(x,\tilde{x})\geq \delta$ implies $F(x,t)<F(\tilde{x},\tilde{t})-\epsilon_{\delta}$. Take an interval $I$ centered at $\tilde{t}$ such that for $t\in I$, $F(\tilde{x},t)>F(\tilde{x},\tilde{t})-\epsilon_{\delta}$. If $q_1=\Pi^{-1}(\bar{X_1})$ and $q_2=\Pi^{-1}(\bar{X_2})$, let $Q$ be a finite word such that $X_2QX_1$ is admissible (where we can set $Q=\emptyset$ if $X_1=X_2$). As $\tilde{x}$ is not periodic then we can take 
$$0<\delta<\min \{d(\tilde{x},\{\varphi ^j(\tilde{x}):j\neq 0 \}),d(\tilde{x},\mathcal{O}(\Pi^{-1}(\bar{X}_2Q\bar{X}_1))) \}$$
(note that $\tilde{x}\notin \mathcal{O}(\bar{X}_2Q\bar{X}_1 )$). Take $j> 0$ and $H$ such that $X_1HX_2$ is admissible and $\Pi(\varphi^{-j}(\tilde{x}))=X_1^{\infinity};HX_2^{\infinity}$. By Lemma \ref{lemarota}, for $t\in I$ we can define the point 
$$\tilde{x}_t=\Pi^{-1}(X_1^{\infinity};HX_2^{n_{1,1}}QX_1^{n_{2,1}}HX_2^{n_{1,2}}QX_1^{n_{2,2}}H\dots),$$ 
where the sequences $\{n_{1,s}\}_{s\in\mathbb{N}},\{n_{2,s}\}_{s\in\mathbb{N}}$ are taken in such a way that:
\begin{itemize}
    \item $n_{1,s}\rightarrow\infinity$, $n_{2,s}\rightarrow\infinity$
    \item $\bigcirc_{\tilde{s}=0}^{(\Sigma_{i=1}^{s}n_{1,i})\cdot\abs{X_2}+(\Sigma_{i=1}^{s}n_{2,i})\cdot\abs{X_1}+s\cdot(\abs{Q}+\abs{H})}R_{\varphi^{\tilde{s}}(\tilde{x}_t)}(t)\rightarrow (\bigcirc_{\tilde{s}=0}^{j-1} R_{\varphi^{\tilde{s}-j}(\tilde{x})})^{-1}(t).$
\end{itemize} 
If $\bigcirc_{s=0}^{k_1-1} R_{\varphi^s(q_1)}$ has irrational rotation number, we can take $n_{1,s}=s$ and if $\bigcirc_{s=0}^{k_2-1} R_{\varphi^s(q_2)}$ has irrational rotation number, we can take $n_{2,s}=s$. With this conditions, we have that
$$\ell_{\Phi,F}(\tilde{x}_t,t)= \limsup_{n\rightarrow\infinity} F(\varphi^n(\tilde{x}_t),\bigcirc_{\tilde{s}=0}^{n-1} R_{\varphi^{\tilde{s}}(\tilde{x}_t)}(t))=F(\tilde{x},t)$$
since for $n_{1,s},n_{2,s}$ big, the value of $F$ in the positions corresponding to any position different to the position $j$ of $H$ (here the angle in $\mathbb{S}^1$ is not important) gives value less than $F(\tilde{x},\tilde{t})-\epsilon_{\delta}$ and in the positions corresponding to the position $j$ of $H$ the value of $F$ is converging to $F(\tilde{x},t)$:
\begin{eqnarray*}
&&\lim _{s\rightarrow\infinity}F(\Phi^{j+(\Sigma_{i=1}^{s}n_{1,i})\cdot\abs{X_2}+(\Sigma_{i=1}^{s}n_{2,i})\cdot\abs{X_1}+s\cdot(\abs{Q}+\abs{H})}(\tilde{x}_t,t))\\&=&\lim _{s\rightarrow\infinity}F(\varphi^j(\varphi^{-j}(\tilde{x})),\bigcirc_{\tilde{s}=0}^{j+(\Sigma_{i=1}^{s}n_{1,i})\cdot\abs{X_2}+(\Sigma_{i=1}^{s}n_{2,i})\cdot\abs{X_1}+s\cdot(\abs{Q}+\abs{H})} R_{\varphi^{\tilde{s}}(\tilde{x}_t)}(t))\\&=&F(\tilde{x},\bigcirc_{\tilde{s}=0}^{j-1}R_{\varphi^{\tilde{s}-j}(\tilde{x})}((\bigcirc_{\tilde{s}=0}^{j-1} R_{\varphi^{\tilde{s}-j}(\tilde{x})})^{-1}(t)))\\&=&F(\tilde{x},t)
\end{eqnarray*}
As $F(\tilde{x},I)=\{F(\tilde{x},t):t\in I \}$ is a non-trivial interval, this finishes the proof of the proposition (here we used again that the point of maximum is unique and that the function $t\rightarrow F(\tilde{x},t)$ is continuous).

\end{proof}

At this point, we are ready to prove Theorem \ref{principal1}. 

\begin{proof}
The proof is a direct consequence of the previous propositions: in this case for every $x\in S$, $R_x={\text{rot}}_{\alpha}$, and then given a periodic point $p\in\Lambda$ of $\varphi$ of period $k\in\mathbb{N}$, the rotation number of $\bigcirc_{s=0}^{k-1} R_{\varphi^s(q)}=\bigcirc_{s=0}^{k-1} {\text{rot}}_{\alpha}={\text{rot}}_{k\cdot \alpha}$ is irrational and the conjugacy is given by $g=\text{id}_{\mathbb{S}^1}$, the identity function  of $\mathbb{S}^1$ which is a diffeomorphism. Then, if $F\in \mathcal{R}_{\varphi,\Lambda}$ and $M_F(\Lambda\times \mathbb{S}^1)=\{(\tilde{x},\tilde{t})\}$, to see that there is a sequence of non-trivial intervals contained in $\mathcal{L}_{\Lambda\times \mathbb{S}^1,\varphi\times R_{\alpha},F}$ converging to $F(\tilde{x},\tilde{t})$, we can take any periodic point $q\in \Lambda$ of $\varphi$ such that $\tilde{x}\notin \mathcal{O}(q)$. If $\tilde{x}$ is not periodic, as $\tilde{x}$ is a corner of rectangle of $\Lambda$, we can find $q_1, q_2\in \Lambda$ periodic points of $\varphi$ such that $\tilde{x}\in W^s(q_2)\cap W^u(q_1)$. Then, we can consider any of them and conclude that a non-trivial interval containing $F(\tilde{x},\tilde{t})$ is contained in $\mathcal{L}_{\Lambda\times \mathbb{S}^1,\varphi\times R_{\alpha},F}$.
\end{proof}

\begin{corollary}\label{where}
Let $\Lambda$ be a horseshoe of $\varphi:S\rightarrow S$ with $HD(\Lambda)<1$ and $\alpha\in \mathbb{R}\setminus \mathbb{Q}$ be fixed. If $\varphi\times {\text{rot}}_{\alpha}:S\times \mathbb{S}^1\rightarrow S\times \mathbb{S}^1$ is given by $(\varphi\times {\text{rot}}_{\alpha})(x,t)=(\varphi(x),{\text{rot}}_{\alpha}(t))$, then for each $F\in \mathcal{R}_{\varphi,\Lambda}$:
 $$\mathcal{L}_{\Lambda\times \mathbb{S}^1,\varphi\times {\text{rot}}_{\alpha},F}\neq\mathcal{L}_{\Lambda,\varphi,f_F} \ \ \text{and} \ \ \mathcal{M}_{\Lambda\times \mathbb{S}^1,\varphi\times {\text{rot}}_{\alpha},F}\neq\mathcal{M}_{\Lambda,\varphi,f_F}.$$ 
\end{corollary}

\begin{proof}
In this case, as $f_F$ is Lipschitz, one has 
$$HD(\mathcal{L}_{\Lambda,\varphi,f_F})\leq HD(f_F(\Lambda))\leq HD(\Lambda)<1$$ 
and as $int( \mathcal{L}_{\Lambda\times \mathbb{S}^1,\varphi\times {\text{rot}}_{\alpha},F})\neq \emptyset$, we have $HD(\mathcal{L}_{\Lambda\times \mathbb{S}^1,\varphi\times {\text{rot}}  _{\alpha},F})=1$. The argument for the Markov spectrum is similar once $\mathcal{L}_{\Lambda\times \mathbb{S}^1,\varphi\times {\text{rot}}_{\alpha},F}\subset \mathcal{M}_{\Lambda\times \mathbb{S}^1,\varphi\times {\text{rot}}_{\alpha},F}.$
\end{proof}

\subsection{Theorem on the intersection of the Lagrange spectrum with half-lines}  Let $\varphi_0$ be a smooth diffeomorphism of a surface $S$ possessing a horseshoe $\Lambda_0$. Denote by $\mathcal{U}$ a  $C^{2}$-neighborhood of $\varphi_0$ in the space $\textrm{Diff}^{2}(S)$ of smooth diffeomorphisms of $S$ such that $\Lambda_0$ admits a continuation $\Lambda$ for every $\varphi\in\mathcal{U}$. 

\begin{lemma}\label{discontinuities}
If $\mathcal{U}\subset\textrm{Diff}^{2}(S)$ is sufficiently small, there exists a residual subset $\mathcal{U}^*\subset \mathcal{U}$ such that given $r\geq2$ and $\varphi\in\mathcal{U}^*$ there exits a $C^r$-residual set $\mathcal{P}_{\varphi,\Lambda}\subset C^r(S,\mathbb{R})$ with the following properties:
\begin{itemize}
    \item Given $\varphi\in \mathcal{U}^*$, for each subhorseshoe $\tilde{\Lambda}\subset \Lambda$ and $f\in \mathcal{P}_{\varphi,\Lambda}$
    $$\min \{1, HD(\tilde{\Lambda}) \}=HD(\ell_{\varphi,f}(\tilde{\Lambda})).$$
    \item Given $\varphi\in \mathcal{U}^*$ and $f\in \mathcal{P}_{\varphi,\Lambda}$, then 
    $$\mathcal{L}^{'} _{\Lambda,\varphi, f}=\{x: x\ \textrm{is an accumulation point of}\ \mathcal{L}_{\Lambda,\varphi, f}\}\neq \emptyset$$ and if $c_{\varphi,f}=\min \mathcal{L}^{'} _{\Lambda,\varphi, f}$, we have 
    $$c_{\varphi,f}=\max \{t\in \mathbb{R}:L_{\Lambda,\varphi,f}(t)=0 \}=\max \{ t\in \mathbb{R}:HD(\Lambda_t)=0\},$$
    where for $t\in \mathbb{R}$, $L_{\Lambda,\varphi,f}(t)=HD(\mathcal{L}_{\Lambda,\varphi,f}\cap(-\infinity,t))$ and $\Lambda_t=\{x\in \Lambda: m_{\varphi,f}(x)\leq t\}$.
    \item Given $\varphi\in \mathcal{U}^*$, $f\in \mathcal{P}_{\varphi,\Lambda}$ and $\epsilon>0$, there exists a non-trivial subhorseshoe $\tilde{\Lambda}\subset \Lambda$ such that $\max f|_{\tilde{\Lambda}}<c_{\varphi,f}+\epsilon.$
    \item Given $\varphi\in \mathcal{U}^*$, if $\tilde{\Lambda}\subset \Lambda$ is a subhorseshoe and $f\in \mathcal{P}_{\varphi,\Lambda}$, then there exists a unique point $\tilde{x}\in \tilde{\Lambda}$ such that $\max f|_{\tilde{\Lambda}}=f(\tilde{x})$ and $Df_{\tilde{x}}(e^{s,u}_{\tilde{x}})\neq 0$.   
\end{itemize}
    
\end{lemma}
\begin{proof}
These results are in section $3$ of \cite{C1}.
\end{proof}

\begin{lemma}\label{Cr-1}
The set $R$ of the $g\in C^r(\mathbb{S}^1,\mathbb{R})$ such that there is a unique point $\tilde{t}\in \mathbb{S}^1$ such that $\max g=g(\tilde{t})$ and $\frac{d^2}{dt^2}g(\tilde{t})<0$, is open and dense.
\end{lemma}
\begin{proof}
The proof uses similar ideas as the proofs of Proposition \ref{Cr-1}, Lemma \ref{denso} and Lemma \ref{open}: for $g$ close enough to $g_0\in R$, the points of maximum are close to the points of maximum of $g_0$ and we have exactly one because of Rolle's theorem and because the second derivative is negative. The set is dense because we can always add to any $g\in C^r(\mathbb{S}^1,\mathbb{R})$ a small perturbation close to $0$ in order to have a unique point of maximum $\tilde{t}$ and then do another perturbation by adding a non-positive small perturbation of the function $0$ that close to $\tilde{t}$ takes the form $-a(t-\tilde{t})^2$ where $a>0$ is small. 
  
\end{proof}

Using the sets of the previous lemmas, we have the following result

\begin{theorem}\label{principal0} 
Fix $\varphi\in\mathcal{U}^*$. Given $(f,g)\in \mathcal{P}_{\varphi,\Lambda}\times R$, define the function $F_{fg}:S\times \mathbb{S}^1\rightarrow \mathbb{R}$ for $(x,t)\in S\times \mathbb{S}^1$ by $F_{fg}(x,t)=(f(x)-c_{\varphi,f})\cdot (g(t)-\min g)$. Then, for 
$\alpha \in \mathbb{R}\setminus \mathbb{Q}$, the function $L_{\Lambda\times \mathbb{S}^1,\varphi\times {\text{rot}}_{\alpha},F_{fg}}(s)=HD(\mathcal{L}_{\Lambda\times \mathbb{S}^1,\varphi\times {\text{rot}}_{\alpha},F_{fg}}\cap (-\infinity,s))$ is given by
\begin{equation*}
L_{\Lambda\times \mathbb{S}^1,\varphi\times {\text{rot}}_{\alpha},F_{fg}}(s)=
    \begin{cases}
        0 & \text{if } s\leq 0\\
        1 & \text{if } s>0.
    \end{cases}
\end{equation*}
In fact, $\mathcal{L}_{\Lambda\times \mathbb{S}^1,\varphi\times {\text{rot}}_{\alpha},F_{fg}}\cap (-\infinity,s)=\emptyset$ if $s\leq 0$ and $\mathcal{L}_{\Lambda\times \mathbb{S}^1,\varphi\times {\text{rot}}_{\alpha},F_{fg}}\cap (-\infinity,s)$ has non-empty interior if $s>0$.
\end{theorem}

\begin{proof}
Given $(x,t)\in S\times \mathbb{S}^1$ note that it is always possible to find a sequence $\{n_k \}_{k\in\mathbb{N}}$ such that $n_k\to\infinity$ and $R^{n_k}_{\alpha}(t)\to \bar{t}_0$, where $\bar{t}_0$ is any point satisfying $g(\bar{t_0})=\min g$. Therefore, we have $\ell_{\varphi\times {\text{rot}}_{\alpha}, F_{fg}}(x,t)\geq 0$ because $\lim\limits_{n\to\infinity}(f(\varphi^{n_k}(x))-c_{\varphi,f})\cdot(g(R^{n_k}(t))-\min g)=0$, since $\{f(\varphi^{n_k}(x))-c_{\varphi,f}\}_{k\in \mathbb{N}}$ is bounded. Then $\mathcal{L}_{\Lambda\times \mathbb{S}^1,\varphi\times {\text{rot}}_{\alpha},F_{fg}}\cap (-\infinity,0)=\emptyset$ and if $s\leq 0$, $L_{\Lambda\times \mathbb{S}^1,\varphi\times {\text{rot}}_{\alpha},F_{fg}}(s)=0$ (the equality remains true if $s=0$ because the function $L_{\Lambda\times \mathbb{S}^1,\varphi\times {\text{rot}} _{\alpha},F_{fg}}$ is continuous on the left). Let now $s>0$, as $\max g-\min g>0$, Lemma \ref{discontinuities} let us find a non-trivial subhorseshoe $\tilde{\Lambda}\subset \Lambda$  such that $\max f|_{\tilde{\Lambda}}<c_{\varphi,f}+\frac{s}{\max g-\min g}$. Note that $\max f|_{\tilde{\Lambda}}>c_{\varphi,f}$ since, in other case, $\ell_{\varphi,f}(\tilde{\Lambda})\subset (-\infinity,c_{\varphi,f}]$ and then
$$0<\min \{1, HD(\tilde{\Lambda}) \}=HD(\ell_{\varphi,f}(\tilde{\Lambda}))\leq HD(\mathcal{L}_{\Lambda,\varphi,f}\cap (-\infinity,c_{\varphi,f}))=L_{\Lambda,\varphi,f}(c_{\varphi,f})=0.$$
Additionally, $F_{fg}\in \mathcal{R}_{\varphi,\Lambda}$: we have that $F_{fg}|_{\tilde{\Lambda}\times \mathbb{S}^1}$ has only one point of maximum at $(\tilde{x},\tilde{t})$ where $\tilde{x}$ is the unique point of maximum of $f|_{\tilde{\Lambda}}$ and $\tilde{t}$ is the unique point of maximum of $g$ and
$$D(F_{fg})_{(\tilde{x},\tilde{t})}(e^{s,u}_{\tilde{x}})=Df_{\tilde{x}}(e^{s,u}_{\tilde{x}})\cdot (\max g-\min g)\neq 0,$$ 
$$\frac{\partial^2F_{fg}}{\partial t^2}(\tilde{x},\tilde{t})=(\max f|_{\tilde{\Lambda}}-c_{\varphi,f})\cdot \frac{d^2g}{d t^2}(\tilde{t})<0.$$
As $\max F_{fg}|_{\tilde{\Lambda}\times \mathbb{S}^1}=(\max f|_{\tilde{\Lambda}}-c_{\varphi,f})\cdot(\max g-\min g)<s$, Theorem $\ref{principal1}$ let us conclude that there is a non-trivial interval $I$ such that 
$$I\subset \mathcal{L}_{\tilde{\Lambda}\times \mathbb{S}^1,\varphi\times {\text{rot}}_{\alpha},F_{fg}}\subset\mathcal{L}_{\Lambda\times \mathbb{S}^1,\varphi\times {\text{rot}}_{\alpha},F_{fg}}\cap (-\infinity,\max F_{fg}|_{\tilde{\Lambda}\times\mathbb{S}^1})\subset\mathcal{L}_{\Lambda\times \mathbb{S}^1,\varphi\times {\text{rot}}_{\alpha},F_{fg}}\cap (-\infinity,s).$$
From this it follows that $L_{\Lambda\times \mathbb{S}^1,\varphi\times {\text{rot}} _{\alpha},F_{fg}}(s)=1$ as we wanted to see. 
\end{proof}

\begin{remark}
 If $(f,g)\in \mathcal{P}_{\varphi,\Lambda}\times R$ and 
 $$\mathcal{L}^{'}_{\Lambda\times \mathbb{S}^1,\varphi\times {\text{rot}}_{\alpha},F_{fg}}=\{x: x\ \textrm{is an accumulation point of}\ \mathcal{L}_{\Lambda\times \mathbb{S}^1,\varphi\times {\text{rot}}_{\alpha},F_{fg}}\},$$ then $0=\min \mathcal{L}^{'}_{\Lambda\times \mathbb{S}^1,\varphi\times {\text{rot}}_{\alpha},F_{fg}}=\min \mathcal{L}_{\Lambda\times \mathbb{S}^1,\varphi\times {\text{rot}}_{\alpha},F_{fg}}$ is accumulated on the right by non-trivial intervals contained in $\mathcal{L}_{\Lambda\times \mathbb{S}^1,\varphi\times {\text{rot}}_{\alpha},F_{fg}}.$
\end{remark}

\begin{remark}
 Observe that, once we fix some $\varphi\in\mathcal{U}^{*}$, the function $c_{\varphi}:\mathcal{P}_{\varphi,\Lambda}\rightarrow\mathbb{R}$ defined by $c_{\varphi}(f)=c_{\varphi,f}$ is continuous: Given $f\in \mathcal{P}_{\varphi,\Lambda}$ and $t\in \mathbb{R}$, denote by $\Lambda^f_t$ to the set $\Lambda^f_t=\{x\in \Lambda: m_{\varphi,f}(x)\leq t\}$. Given $f_1,f_2\in\mathcal{P}_{\varphi,\Lambda}$ such that $d(f_1,f_2)<\frac{\epsilon}{2}$, we have for $x\in \Lambda^{f_1}_t$ and $n\in \mathbb{Z}$, that $f_2(\varphi^n(x))-\frac{\epsilon}{2}<f_1(\varphi^n(x))\leq t$ and then $x\in \Lambda^{f_2}_{t+\frac{\epsilon}{2}}$. By symmetry, one concludes that
 $$\Lambda^{f_2}_{t-\frac{\epsilon}{2}}\subset\Lambda^{f_1}_t\subset\Lambda^{f_1}_{t+\frac{\epsilon}{2}}\subset\Lambda^{f_2}_{t+\epsilon}.$$
 If we take $t=c_{\varphi,f_1}$, we have that $HD(\Lambda^{f_2}_{c_{\varphi,f_1}-\frac{\epsilon}{2}})=0$ and $HD(\Lambda^{f_2}_{c_{\varphi,f_1}+\epsilon})\geq HD(\Lambda^{f_1}_{c_{\varphi,f_1}+\frac{\epsilon}{2}})\\ >0$. Therefore Lemma \ref{discontinuities} let us conclude that $c_{\varphi,f_1}-\frac{\epsilon}{2}\leq c_{\varphi,f_2}<c_{\varphi,f_1}+\epsilon$ and from this $\abs{c_{\varphi,f_1}-c_{\varphi,f_2}}<\epsilon$. In the same way, the function $g\mapsto \min g$ is continuous: If $d(g_1,g_2)\leq \epsilon$ then for all $t\in \mathbb{S}^1$, $\min g\leq g_1(t)\leq g_2(t)+\epsilon$. From this $\min g_1\leq \min g_2+\epsilon$, and by symmetry $\min g_2\leq \min g_1+\epsilon$, that is $\abs{\min g_1-\min g_2}\leq \epsilon$. Thus the function $F:\mathcal{P}_{\varphi,\Lambda}\times R\rightarrow C^r(S\times \mathbb{S}^1,\mathbb{R})$ given by $F(f,g)=F_{fg}$ is continuous.
 \end{remark}

\end{document}